\documentclass[12pt]{amsart}
\usepackage{txfonts}
\usepackage{bbm}
\usepackage{mathrsfs}
\usepackage{amssymb}
\usepackage{amssymb,amsmath,amsthm,amscd}
\usepackage[all]{xy}

\usepackage{graphicx}

\numberwithin{equation}{section}

\newtheoremstyle{fancy1}{10pt}{10pt}{\itshape}{12pt}{\textsc\bgroup}{.\egroup}{8pt}{
}
\newtheoremstyle{fancy2}{10pt}{10pt}{}{12pt}{\itshape}{.}{8pt}{ }

\theoremstyle{fancy1}

\newtheorem{cor}[equation]{Corollary}
\newtheorem{lem}[equation]{Lemma}
\newtheorem{prop}[equation]{Proposition}
\newtheorem{thm}[equation]{Theorem}
\newtheorem{problem}[equation]{Problem}
\newtheorem{main}{Theorem}
\newtheorem*{main*}{Theorem}
\newtheorem*{conjecture}{Conjecture}
\newtheorem*{cor*}{Corollary}

\theoremstyle{fancy2}

\newtheorem{rem}[equation]{Remark}
\newtheorem*{rem*}{Remark}
\newtheorem{example}[equation]{Example}

\newcommand{\cref}[1]{Corollary~\ref{#1}}

\newcommand{\Sph}{\mathbb{S}}
\newcommand{\Disc}{\mathbb{D}}
\newcommand{\W}{\mathsf{W}}

\newcommand{\N}{\mathsf{N}}

\newcommand{\RP}{\mathbb{R\mkern1mu P}}

\newcommand{\C}{{\mathbb{C}}}
\newcommand{\R}{{\mathbb{R}}}
\newcommand{\Z}{{\mathbb{Z}}}

\newcommand{\E}{\ensuremath{\operatorname{\mathsf{E}}}}
\newcommand{\F}{\ensuremath{\operatorname{\mathsf{F}}}}
\newcommand{\G}{\ensuremath{\operatorname{\mathsf{G}}}}
\newcommand{\GL}{\ensuremath{\operatorname{\mathsf{GL}}}}
\newcommand{\D}{\ensuremath{\operatorname{\mathsf{D}}}}

\renewcommand{\O}{\ensuremath{\operatorname{\mathsf{O}}}}

\renewcommand{\S}{\ensuremath{\operatorname{\mathsf{S}}}}

\newcommand{\A}{\ensuremath{\operatorname{\mathsf{A}}}}
\newcommand{\CC}{\ensuremath{\operatorname{\mathsf{C}}}}
\newcommand{\B}{\ensuremath{\operatorname{\mathsf{B}}}}

\renewcommand{\r}{\ensuremath{\operatorname{\mathsf{r}}}}

\newcommand{\w}{\ensuremath{\operatorname{\mathsf{w}}}}

\def\con#1=#2(#3){#1 \equiv #2 \bmod{#3}}

\newcommand{\cl}{\ensuremath{\operatorname{cl}}}

\newcommand{\no}{\noindent}

\renewcommand{\F}{\mathsf{F}}

\begin{document}

\title{Reflection groups in non-negative curvature}

\author{Fuquan Fang}
\address{Department of Mathematics\\Capital Normal University\\
Beijing 100048\\China } \email{fuquan\_fang@yahoo.com }


\author{Karsten Grove}
\address{Department of Mathematics\\
University of Notre Dame\\
      Notre Dame, IN 46556\\USA
      }
\email{kgrove2@nd.edu }

\thanks{The first named author is supported in part by an NSFC key grant and University of Notre Dame.
The second named author is supported in part by the NSF and a Humboldt award.  He also wants to thank the MPI and the University of Bonn for hospitality}

\maketitle



\begin{abstract} We provide an equivariant description / classification of all complete (compact or not) non-negatively curved manifolds $M$ together with a co-compact action by a reflection group $\W$, and moreover, classify such $\W$. In particular, we show that the building blocks consist of the classical constant curvature models and generalized open books with non negatively curved bundle pages, and derive a corresponding splitting theorem for the universal cover.
\end{abstract}

\bigskip


The theory of discrete groups of motions
generated by reflections has a long history (cf., e.g., \cite{Bo}) going back to the study of planar regular
polygons and space polyhedra.  It's impact on the modern development of Lie theory, and symmetric spaces going back to E. Cartan and W. Killing is well known.

Much of the work on reflection groups has been focussed on constant curvature spaces. Here, the euclidean and spherical cases are well understood ultimately due to the works of H. S. M. Coxeter \cite{Co}.
 In the hyperbolic case the situation is very different.  A complete classification of reflection groups in the hyperbolic plane was achieved by Poincar\'e \cite{Po1} (cf. also von Dyck \cite{Dy}), and in the hyperbolic $3$-space  by Andreev \cite{An}
 , whereas hyperbolic reflection groups in higher dimensions are very rich and far from being classified. A surprising theorem of Vinberg \cite{Vi} asserts there are no co-compact hyperbolic reflection group in dimensions $\ge 30$.

\smallskip

Here we deal with general Riemannian manifolds  with variable, but non-negative (sectional) curvature equipped with a co-compact proper action by a discrete reflection group. Our results provide an essentially complete understanding of these objects.


\smallskip

In contrast to the classical framework discussed above, but motivated by applications to polar actions (like the one in \cite{FGT}), a \emph{reflection} is nothing but an isometric involution whose fixed point set has a component of codimension one, called a \emph{mirror}. Most subtleties caused by this generality evaporate when passing to a canonical finite cover (see Proposition \ref{coxeter}).

 The following simple example is at the core of our work: Consider a reflection $\r: M \to M$ whose mirror $\Lambda$ separates $M$. From the Cheeger-Gromoll soul construction it follows that $M$ is the double of a disc bundle $\Disc(\nu)$. Note that this double can also be described as the sphere bundle $\Sph(\nu \oplus \varepsilon)$ ($\varepsilon$ is the trivial line bundle), as well as an \emph{open book} with two \emph{pages} $\Disc(\nu)$, i.e., parametrized by $\Sph^{0}$, having common boundary, the \emph{binding} $\Lambda$.

 It turns out that a natural generalization of the above \emph{open book} type of manifold, together with the \emph{classical space forms} constitute the \emph{building blocks} needed in general. To explain the appearance of building blocks, we say that the action $\W\times M \to M$ is \emph{decomposable} if the orbitspace $M/\W$ metrically is a finite quotient of a product, and \emph{indecomposable} otherwise. With this terminology one of our main results is the following \emph{Rigidity Theorem}

 \begin{main}
 A nonnegatively curved manifold $M^n$ with an indecomposable cocompact action by a reflection group $\W$ is isometric to  either $\Bbb R^n$, or $\Bbb T^n$, or equivariantly diffeomorphic to either $\Sph^n$, or $\Bbb {RP}^n$ with a linear action, unless all mirrors in $M$ meet.
 \end{main}

Here the spherical case relies on showing that the orbitspace is a simplex (cf. section 2), whereas the part where the universal cover of $M$ is non-compact also relies on Cheeger- Gromoll splitting results for cocompact actions and for compact manifolds with infinite fundamental group, as well as on  Bieberbach's celebrated Theorem (cf. section 4). Recall, that by the latter, any compact flat manifold is
finitely covered by a  flat torus, i.e., $M=\Bbb T^n/\G$, where
$\G\subset \O(n)$ is the holonomy. In particular, Theorem A shows that the
holonomy group $\G$ must be trivial when the action is indecomposible. We will prove, moreover, that
if the orbit space splits as a metric product of eucledian
simplices, then $\Bbb T^n/\G$ must be an iterated torus bundle, with
holonomy group $\G$ a very special elementary abelian $2$-group in
$\GL(\Bbb Z, n)$ (see Corollary \ref{holonomy}). The Klein bottle serves as the simplest example.

 \smallskip

 To describe the structure that arises when all mirrors meet consider the following generalizations of the open book with two pages discussed above:

 \smallskip
 {\sc Model Examples}. Let $\rho$ be a linear representation  of a finite Coxeter group $\W$ on $\R^k$, and $\nu$  a smooth vector bundle with base space $S$.  The obvious action by $\W$ on the bundle $\nu \oplus \varepsilon
^{k}$, where $\varepsilon$ is a trivial line bundle, induces an action by $\W$ on the total space of the sphere bundle $M_{\rho,\nu} = \Sph (\nu \oplus \varepsilon ^{k}) =:\Sph(\nu,\rho)$. Note that, this action has $k$ mirrors, whose intersection is $B: = \Sph (\nu) \subset \Sph (\nu \oplus \varepsilon ^{k})$, and ``normal" to $B$ the action is $\rho$. Note also that the equivariant projection $\nu \oplus \varepsilon^{k} \to \R^k$ induces an equivariant map $L: \Sph (\nu \oplus \varepsilon ^{k}) \to \Disc^k$, with $L^{-1} (0) = B$ and $L^{-1}([0,1]x) = P$ diffeomorphic to $\Disc(\nu)$ for any $x \in \partial \Disc^k = \Sph^{k-1}$. For this reason we call $M_{\rho,\nu}$ a $k-1$ dimensional \emph{open book} with binding $B$ and \emph{pages} $P$, parametrized by $\Sph^{k-1}$.

In general, given $\ell$ linear representations $\rho_i$ of finite Coxeter groups $\W_i$ on $\R^{k_i}$, and $\ell$ smooth vector bundles $\nu_i$ with base $S$. The obvious $\W = \W_1\times \ldots\times\W_{\ell}$ action  on the product of the bundles $\nu_i \oplus \varepsilon^{k_i}$ induces a $\W$ action on the \emph{fiber product}, $M_{\bar{\rho}, \bar{\nu}} := \Sph(\bar{\nu},\bar{\rho})$ of the sphere bundles $\Sph(\nu_i \oplus \varepsilon^{k_i})$, i.e., the pull back  by the diagonal map $\Delta: S \to S \times \ldots \times S$ of the product of the sphere bundles $\Sph(\nu_i \oplus \varepsilon^{k_i})$. As in the case of a single representation and bundle, there is a canonically associated $\W$ equivariant map $L : M_{\bar{\rho}, \bar{\nu}} \to \Disc^{k_1}\times \dots \times \Disc^{k_{\ell}}$ where $B = L^{-1}(0, \ldots,0)$ is the intersection of all mirrors for $\W$, and $P = L^{-1}([0,1]x_1, \ldots, [0,1]x_{\ell})$ for any $\bar{x} \in \Sph^{k_1-1} \times \ldots \times \Sph^{k_{\ell}-1}$ is a manifold with corners diffeomorphic to the fiber product $\Disc(\bar{\nu})$ of the disc bundles $\Disc({\nu_i})$. We say that  $M_{\bar{\rho}, \bar{\nu}}$ is an \emph{iterated open book} with \emph{pivot binding} $B$ and \emph{pages} $P$.

\smallskip

 Using this terminology we have the following general \emph{Structure Theorem} when all mirrors meet.

\begin{main}
A  compact nonnegatively curved manifold $M$ with reflection group $\W$, all of whose mirrors meet admits a finite cover $M'$ which is equivariantly equivalent to an (iterated) open book $M_{\bar{\rho},\bar{\nu}}$, with pages a non negatively curved (fiber product) disc bundle $\Disc(\bar{\nu})$.
\end{main}

For more details including further restrictions on the metric on the pages, we refer to section 3, in particular Theorems \ref{open book} and \ref{iterated open book} and the description about additional geometric structure in the form of the presence of \emph{spherical heavens} of souls in the spirit of Yim's work \cite{Yi2}.  Also, \emph{conversely}, using a construction due to Guijarro \cite{Gu}, it follows that an (iterated) open book with the given data has an invariant metric with nonnegative curvature.

\smallskip
When passing to the universal cover, the above results in particular lead to the following general \emph{Splitting Theorem}

\begin{main}\label{splitting}
Let $M$ be a complete non negatively curved manifold with co-compact reflection group $\W$. Then the lifted reflection group $\hat{\W}$ on the universal cover $\tilde M$ is a product of Coxeter groups,
\begin{center}
$\hat{\W} = \hat{\W}_0 \times \displaystyle \prod_{i = 1}^ {\ell -1} \hat \W_i \times  \  \hat{\W}_\ell$,
\end{center}

\no where $\hat \W_0$ is affine, and the remaining factors are spherical. Correspondingly, $\tilde M$ admits a $\hat{\W}$ invariant metric splitting,
\begin{center}
$\tilde M = \R^k  \times \displaystyle \prod_{i = 1}^{\ell -1}  \Sph^{k_i}  \times \Theta_\ell \times N$,
\end{center}

\no where $N$ can be any simply connected compact manifold of nonnegative curvature on which all $\hat \W_i$ act trivially, $\Sph ^{k_i}$ is a non negatively curved standard sphere with a linear $\hat \W_i$ action, and  $\Theta_\ell$ is a compact simply connected non-negatively curved \emph{(iterated) open book}.
\end{main}

\smallskip
As a consequence we derive the following \emph{Group Structure Theorem},

\begin{cor*}
A group $\W$ is a co-compact reflection group of a complete non negatively curved  manifold if and only if
\begin{center}
$\W\cong \hat \W_0\times \cdots \times \hat \W_\ell /\N$,
\end{center}

\no  where $\hat \W_0$ is an affine Coxeter group, $\hat \W_i$, $1\leq i\leq \ell$, is a
spherical Coxeter group, and $\N\lhd \hat \W$ a normal subgroup isomorphic to a product of a torsion free lattice and an elementary abelian $2$-group.
\end{cor*}


\smallskip

As indicated earlier, aside from obviously being of interest on its own, understanding reflections groups in nonnegative curvature provides the first step in understanding so-called polar actions on such manifolds (cf. \cite{FGT}, where a complete classification of polar actions in positive curvature, of cohomogeneity at least two, was carried out). The reason is that  so-called sections of a polar action are non-negatively curved manifolds with a reflection group.   Basic examples of such actions are provided by compact Lie groups with adjoint actions, where the sections are the maximal tori.  Note, that in this context, it is potentially important to include non compact reflection manifolds, since a priory it is not known if sections are compact even when the polar manifold is.

In general, there will be  no classification like in \cite{FGT} because of the presence of open books as sections. In fact, potentially one might be able to construct new non-negatively curved (polar) manifolds as in the case of cohomogeneity one actions considered in \cite{GZ}, when sections are open books.

Note also, however, that a polar action with open books as sections, should be considered as \emph{reducible}, since the associated reflection group of a section has an \emph{invariant subspace} (a totally geodesic submanifold). Thus, Theorem A, is the key starting point in an analysis of  \emph{irreducible} polar actions on compact simply connected manifolds of nonnegative curvature,  for which the following was proposed in \cite{FGT}:

\begin{conjecture}
An irreducible polar action on a simply connected nonnegatively
curved compact manifold is equivariantly diffeomorphic to a quotient of a polar
action on a symmetric space.
\end{conjecture}

We point out that in the above sense, in fact any polar action on a simply connected compact symmetric space of nonnegative curvature is the quotient of a polar action on a compact Lie group with a biinvariant metric.

The general structure / classification of compact simply connected non negatively curved polar manifolds will be addressed in forthcoming papers.

\smallskip

We conclude the introduction with a short outline of the paper.

In the first section we provide the necessary background for reflection groups in our generality, including the notion of a \emph{Coxeter action}, where the orbit space $M/\W$ is isometric to the closure $C$ of any open chamber $c$, i.e., of a connected component of the set of all mirrors. We analyze the lift to the universal cover and establish the existence of a canonical lift, the \emph{Coxeter cover}, where the action by $\W$ is Coxeter (Proposition \ref{coxeter}).

The overall strategy in our approach is based on the fact that follows from the work of   W\"{o}rner \cite{Wo} that the chamber $C$ for a Coxeter action is a product   $C = C_0 \times C_1\times C_2 \times \ldots \times C_{\ell}$ where $C_0$ is a manifold without boundary (typically a point), and each $C_i$, $i \ge 1$ is a smooth non negatively curved convex manifold with corners, and either (1) $C_i$ has more than $n_i = \dim C_i$ faces, but any $n_i$ faces of $C_i$ meet, or (2) $C_i$ has $k_i\le n_i$ faces and they all meet. In section 2, we show that if there is only one factor and it is of type (1) then $C$ is a simplex. This is then used to prove the spherical part of theorem A (cf. 2.5 and 2.6). The case where there is only one factor in the splitting, but it has type (2) is then handled in section 3. This is where the open book structures appear, from which Theorem B follows.

The starting point in section 4 is the observation that a co-compact action on a noncompact manifold of nonnegative curvature is decomposable unless the manifold is euclidean space, and similarly an action is decomposable on a compact manifold with infinite fundamental group unless it is flat (cf. 4.2 and 4.3). Consequently, the rest of the section deals with reflection groups on flat manifolds, and in particular the flat part of Theorem A follows from \ref{torus}.

Finally, in section 5 we give proofs of Theorem C and Corollary D. 

\smallskip

It is our pleasure to thank Burkhard Wilking for pointing out the Cheeger-Gromoll Isometry Splitting Theorem (Corollary 6.2 in \cite{CG}) to us. Our original proof of Theorem \ref{splitting} for a co-compact reflection group was based on the work of Yim \cite{Yi2} on the heaven of pseudo souls, and Gromov's theorem about groups of polynomial growth.

\section{Preliminaries and the Coxeter cover.}  

Although our focus in this paper is to analyse and describe complete
nonnegatively curved manifolds with co-compact
reflection groups, we begin with a brief review and discussion of general (co-compact)
\emph{reflection groups}, establish notation and derive important
facts about covers. See also \cite{FGT}
and \cite{GZ}, where examples are discussed, as well as \cite{AKLM}.

\medskip

For us, a \emph{reflection} $\r$ on a Riemannian manifold $M$ is an isometric involution, whose fixed point set $M^{\r}$ contains a component $\Lambda$ of codimension $1$. Any such
component $\Lambda$, is called a \emph{mirror} for $\r$. It is sometimes advantageous to label reflections by mirrors, $\Lambda_{\r}$,  keeping in mind that different mirrors may be
mirrors for the same reflection. It is essential for us not to require that mirrors separate $M$ into different components interchanged by the reflection!  Note that the latter, however, is
the case for reflections on a simply connected manifold \cite{Da}.


\medskip

Let $\W \subset \text{Isom}(M)$ be a discrete  closed subgroup of isometries of $M$ generated by all reflections contained in $\W$. We
will call any such group $\W$ a \emph{reflection group} of $M$. An \emph{open chamber} $c \subset M$ is by definition a connected component of the complement of the union of all mirrors  $\mathfrak{M}$ for all reflections in $\W$. Clearly, $\W$ acts transitively on the set of all open chambers. However, the \emph{stabilizer group} $\W_c$ may be non-trivial in general.

We say that

\begin{itemize}
\item
 The  \emph{action} $\W \times M \to M$ is \emph{Coxeter} if $\W_c$ is trivial.
\end{itemize}

 It is well known that the action is Coxeter when $M$ is simply connected \cite{Da} (and in this case $\W$ is a Coxeter group), or when $M$ is a \emph{section} of a polar action on a simply connected manifold (\cite{AT, GZ}).  We will see below in Proposition \ref{coxeter} that  $M$ admits a natural equivariant $\W_c$ cover, $M'$ with a Coxeter action by $\W$. We will refer to this cover as the \emph{Coxeter cover} of $(M, \W)$, or simply of $M$.

\smallskip

The closure $C = \cl(c)$ is called a \emph{closed chamber} or simply a \emph{chamber}, and clearly $M/\W = C/\W_c$. In particular, $M/\W = C$ when the action is Coxeter. Moreover, any point $p$ in the boundary $\partial C = C - c$ of $C$ is in one or more mirrors (at most $\dim M$). Since $\W$ is discrete, it follows that the isotropy group $\W_p$ for any such $p \in \partial C$ is a finite Coxeter group, and locally $C$ is a finite union of strongly convex sets.  A \emph{chamber face} of $C$ is by definition  a component of the intersection  $ C \cap \Lambda$, $ \Lambda \in \mathfrak{M}$, which contains an open subset of $\Lambda$. We can provide each chamber face
  with a label $i\in I$ and  will denote the face by $F_i$ and the
  corresponding reflection by $r_i$. As mentioned above, note though that different
  faces can correspond to the same reflection, i.e., possibly
  $r_i=r_j$. Obviously, $\W_c$ takes chamber faces to chamber faces, the image of which under the projection map $C \to C/\W_c = M/\W$ constitute the \emph{faces} of the orbit space $M/\W$.
  By construction we note that the boundaries $\partial C$ and $\partial (M/\W)$ are the union of chamber faces, respectively of faces. Note that in general, $C$ is not an Alexandrov space, whereas $C/\W_c = M/\W$ is. 

  \bigskip

  We now proceed to investigate natural reflection groups induced from $\W$ to covers of $M$ beginning with the universal cover.

\smallskip
Consider the universal covering map $\pi: \tilde M\to M$, and let $\tilde \W$ be the group acting on $\tilde M$ consisting of all lifts of all elements of $\W$. Clearly $\tilde \W$ fits into an exact sequence

$$1\to \pi_1\to \tilde \W \to \W\to 1,$$

\no where $\pi_1:=\pi _1(M)$.  Note that in general $\tilde\W$ is not a reflection group, and it may not be finitely generated (even when $\W$ is).

 Now let $\hat \W  \lhd  \tilde \W$ be the normal subgroup generated by all reflections in $\tilde \W$. Note that a mirror for any such reflection of $\tilde M$ is a connected component of the lift of a
 mirror in $M$. Since $\tilde M$ is simply connected, $\hat \W$ is a Coxeter group which acts Coxeter on $\tilde M$ with chamber $\tilde C$. Furthermore, $\tilde C = \tilde M/\hat \W$ is simply
 connected (see, e.g., Prop. 2.14 in [1] ).

\smallskip
Since both $\pi_1$ and $\hat \W$ are normal subgroups of $\tilde \W$, so is $\hat \W\cap \pi_1$. Moreover, it follows that ${\hat \W\cap \pi_1}\lhd \hat \W$ and ${\hat \W\cap \pi_1}\lhd \pi_1$, with quotients $\W$ and  $\Gamma : =\pi_1 /{\hat \W\cap \pi_1}$ respectively. We now claim that $\tilde \W/(\hat \W\cap \pi_1)$ is isomorphic to the direct product $\W \times \Gamma$, i.e., we have an exact sequence

$$1\to \hat \W\cap \pi_1\to \tilde \W \to \W \times \Gamma \to 1.$$

\no Indeed, this is an immediate consequence of the following algebraic lemma applied to the quotient $\tilde \W/\hat \W\cap \pi_1$.

\begin{lem}
Assume $\hat \N$ contains two normal subgroups $\N\lhd \hat \N$ and $\G \lhd \hat \N$ such that $\hat \N= \langle \N, \G\rangle$. Then  $\hat \N=\N \times \G$, the direct product, if $\N\cap \G=\{1\}$.
\end{lem}
\begin{proof} By the assumption, conjugation by elements of $\G$ defines a homomorphism $\rho: \G\to \text{Aut}(\N)$. Similarly, conjugation by elements of $\N$ defines a homomorphism $\tau: \N \to \text{Aut}(\G)$. Note that, for any $x\in \N$ and $g\in \G$, we have
$$gxg^{-1}=\rho(g)(x), \text{and} \ x^{-1}gx=\tau (x)(g)$$

Thus $gx=\rho(g)(x)g$ and $gx=x\tau(x)(g)$, and it follows that $\rho(g)(x)g=x\tau(x)(g)$. Hence
$$x ^{-1} \rho(g)(x) = \tau(x)(g) g^{-1}$$
where the left side belongs to $\N$, and right side belongs to $\G$. From the  assumption, $\N\cap \G=\{1\}$, it follows that both are trivial, in other words  both $\rho$ and $\tau$ are trivial, i.e, $\N$ and $\G$ commute.
\end{proof}

Thus, for the induced action by $\W=\hat \W/\hat \W\cap \pi_1$ on $\hat M:= \tilde M/{\hat \W\cap \pi_1}$, a covering space of $M$ with deck transformation group $\Gamma = \pi_1/(\pi_1 \cap \hat \W)$ we have

\begin{cor}
The action by $\W$  on $\hat M$ is Coxeter, it commutes with the $\Gamma $-action, and its chambers $\hat C$ are isometric to $\tilde C$, in particular they are simply connected.
\end{cor}

\begin{proof}
By construction, it is obvious that chambers of $\hat \W$ in $\tilde M$ are projected isometrically onto chambers for $\W$ on $\hat M$ and that $\W$ acts simply transitive on its set of chambers in $\hat M$, i.e., the action is Coxeter.
\end{proof}

Note that in general the stabilizer $\Gamma_{\hat C}$ of a $\W$ chamber in $\hat M$ is non-trivial and acts freely on the chamber. Since the actions commute, this stabilizer is independent of the chamber and is the kernel $\Gamma_0$  of the induced $\Gamma$ action on the set of chambers in $\hat M$. This now leads to our desired ``resolution" $M' = \hat M/\Gamma_0$ of $M$, the $\Gamma': =\Gamma/\Gamma_0 \cong{\W_c}$  \emph{Coxeter cover} of $M$, with chambers $C' := \hat C/\Gamma_0$:

\begin{prop}[Coxeter cover]\label{coxeter}
Any manifold $M$ with reflection group $\W$, admits a commuting lift to a regular $\Gamma'$ cover $M'$ of $M$ with  Coxeter action by $\W$ and $\Gamma' \cong \W_c$.
\end{prop}

\begin{proof}
Again it is clear from the construction that the chambers of the induced $\W$ action on $M'$ are $C'$, and that $\W$ acts simply transitive on its set of chambers. Moreover, $\W$ commutes with the induced action by $\Gamma':= \Gamma/\Gamma_0$ and $\W \cap \Gamma'$ is trivial.

To see that $\Gamma'$ is isomorphic to $\W_c$, note that for any $\gamma' \in \Gamma'$ and any chamber $C'$ there is a unique $\w(\gamma') \in \W$ with $\gamma'(C') = \w(\gamma') (C')$. It follows that $\w(\gamma') \in \W_c$ and the map $\Gamma' \to \W_c$ is clearly a homomorphism. Conversely, given any $\w \in \W_c$ and chamber $C'$ projecting to $C$ there is a unique $\gamma' \in \Gamma'$ such that
$\w(C') = \gamma'(C')$.
\end{proof}

\begin{rem} We remark that $\W$ may not be a Coxeter group. However, $\hat \W$ is a Coxeter group. Hence $\W$ is a quotient group of $\hat \W$ by a normal subgroup. Notice that if $\hat \W$ is an
irreducible spherical Coxeter group of rank at least $3$, then the
normal subgroup is in the center, which is either trivial or $\Bbb
Z_2$. Similarly, if $\hat \W$ is an irreducible affine Coxeter group
of rank at least $3$, then $\hat \W=\Bbb Z^n\rtimes \W_0$, where
$\W_0$ is an irreducible spherical Coxeter group. A normal subgroup
is a sublattice of $\Bbb Z^n$ or an extension of such a sublattice
by  a center $\Bbb Z_2$ in $\W_0$.
\end{rem}

 \begin{rem}\label{w-structure}
 From the structure of fundamental groups of manifolds with nonnegative curvature, we know that both $\pi_1(M)$ as well as $\pi_1(\hat{M}) = \pi_1\cap \hat \W \lhd \pi_1$ are finitely generated, so all groups in the discussion above are finitely generated in our context of nonnegative curvature. 

Note also, that if $C'$ is a simplex (or a product of simplices), which in nonnegative curvature is often the case (cf. the subsequent sections), then  $C' = \tilde C = \hat C$, i.e,  $\Gamma_0 = \{1\}$ and $\hat M$ is the Coxeter cover of $M$. It follows that $\W = \hat{\W}/\pi_1(\hat{M})$ and $\W_c = \pi_1(M)/\pi_1(\hat{M})$. In particular, $\pi_1(M) \lhd \hat{\W}$  if the action is Coxeter.
 \end{rem}

\medskip
Motivated by \ref{coxeter} and the fact that sections of polar actions on simply connected manifolds are always Coxeter,
\begin{itemize}
\item
We will focus our attention to co-compact Coxeter actions throughout,
\end{itemize}
\no  with the exceptions of  \ref{sphere lift}, \ref{torus}, and \ref{holonomy}.

\smallskip

It is important to us that for Coxeter actions, the chambers $C = M/\W$ have a particularly nice structure:

\begin{rem}[Coxeter chamber structure]\label{chamber}By definition, $C\subset M$ is convex, and assuming $\W$ is  finitely generated, its boundary $\partial C = \bigcup_{i\in I} F_i$ is the union of its faces $F_i$, $i \in I:= \{1, \ldots, k\}$, giving rise to a natural stratification of $C$. To describe the stratification, it is convenient to use the following notation: For any subset $J \subset I$ set $\hat{F}_J: =  \bigcap_{i \in I-J} F_i$, and $F_J:= \bigcup _{i \in J} F_i$, i.e., $\hat{F}_J$ is the intersection of faces opposite of $F_J$. Note that for $J_1 \subset J_2$ obviously $F_{J_1} \subset F_{J_2}$, $\hat{F}_{J_2} \subset \hat{F}_{J_1}$,  and $F_I = \partial C$. By convention we set $\hat{F}_I = C$ and $F_{\emptyset} = \emptyset$.

With this notation all strata $\hat{F_J}$ are \emph{locally totally geodesic}. At interior points, the fibers of the normal bundle to  $\hat{F_J}$ is the orbit space of the normal slice representation of its isotropy group $\W_{\hat{F}_{J}} = \W_{I-J}$. Since $\hat{F}_J$ has codimension 1 in $\hat{F}_{J-j}$ for any $j\in J$ it follows that this \emph{normal bundle is flat and trivial}, in fact it is ``spanned'' by parallel fields. In particular, $C$ also has the structure of a \emph{smooth manifold with corners}, i.e., locally diffeomorphic to open balls of $\R_+^n$. We also point out that since the angle between any two faces is at most $\pi/2$, any of the strata $\hat{F}_J$ are \emph{extremal subsets}  of the Alexandrov space $C$, see, e.g., the survey \cite{Pe}.

There are other natural and useful convex domains associated to $C$, namely the so-called \emph{residues} of $C$. Here the $J$-residue of $C$, $J \subset I$ is the set  $\W_JC$, whose boundary is $\W_JF_{I-J}$.
\end{rem}

\smallskip
The above general structure for $C$ is especially useful in the context of nonnegative curvature, since it enables us to employ numerous \emph{strong convexity arguments} throughout. For example the distance function on $C$ to any face $F_i \subset C$ or union of faces $F_J$ (in particular the whole boundary) is \emph{concave}. One is thus in position to apply corresponding \emph{Sharafutdinov retractions} from $C$ to the associated \emph{soul} of $C$ as in the original approaches to open manifolds in \cite{CG} and \cite{Sh} (This procedure even applies to super level sets of these concave functions as long as they have maximal dimension).

\smallskip

In the context described above, the general work of W\"{o}rner \cite{Wo} about the structure of compact Alexandrov spaces with nonnegative curvature and non-empty boundary, as well as Yim's work
\cite{Yi1,Yi2} on the \emph{heaven} of \emph{pseudo-souls} in a complete open manifold $M$ of nonnegative curvature is very useful for us. Here by definition a subset $S\subset M$ is called a
 {\it pseudo-soul} if it is isometric to a soul $S_0\subset M$, and homologous to $S_0$ in $M$.

\section{Equivariant smooth rigidity: Not all faces meet}    

Unless otherwise stated we assume throughout that $M$ is a non negatively curved compact or complete Riemannian $n$ manifold with a co-compact reflection group $\W$ acting in a Coxeter fashion on $M$ with chamber $C$.

\smallskip

We first point out that the maximal number of faces $F_i$ of $C$ having non-empty intersection is $n$. In fact, at a point $p$ of intersection the corresponding faces of the chamber in the unit tangent
sphere  has at most $n$ faces, and in the latter case this is a spherical $(n-1)$ simplex, actually a fundamental domain for the isotropy group $\W_p$ of $\W$ at $p$ \cite{FGT} (for a more general result we refer to \cite{Wi}).
Also note, that if $n$ faces of $C$ have non-empty intersection, then the intersection consists of isolated points. It follows that, either:

\begin{itemize}
\item
 All faces intersect, in which case $C$ has at most $1\le k \le n$ faces, \quad or

\item
 There is a minimal $0 \le k \le n$ such that: There exist $k+1$ faces with empty intersection.
 \end{itemize}

 \medskip

The above discussion applies to general Alexandrov spaces with nonnegative curvature, for which W\"{o}rner \cite{Wo} proved the following \emph{Splitting Theorem}:

\begin{thm}[W\"{o}rner]\label{wo}
A compact $n$-dimensional Alexandrov space $A$, with non-empty boundary, not all of whose faces meet is  isometric to a product $X^{n-k} \times Y^k$ of non negatively curved Alexandrov spaces, where $X$ is isometric to the intersection of $k$ faces of $C$, with $k$ chosen as above.
\end{thm}

\begin{rem}It also follows (cf. \cite{Wo}) that the maximal number of faces of $A$ is $2n$, in which case $A$ is a product of intervals. When applied to $C$, we conclude in particular that $\W$ is generated by $k \le 2n$ elements.
\end{rem}
\medskip

Repeated applications of \ref{wo} above yields a metric splitting of the Coxeter chamber $C$ of the form

\begin{equation}\label{base split}
 C = \Delta_1 \times \ldots \times \Delta_{r} \times \Bbb V \times N
\end{equation}
\no where $N$ is a closed non negatively curved manifold (typically a point), and each of the remaining factors is a smooth non negatively curved convex manifold with corners, and boundary face structure given by

(1) $\Delta_i$ has more than $n_i = \dim \Delta_i$ faces, but any $n_i$ faces of $\Delta _i$ meet,

(2) $\Bbb V$ has $k\le \dim \Bbb V$ faces and they all meet.

\no The presence of a non-trivial $N$ occurs when taking products with a trivial action on $N$ (cf. \ref{factor}).

\smallskip
Our objective in this section is to begin an analysis of the case where $C$ has only one factor, and this factor is of the first kind $\Delta$. We will refer to this as the \emph{maximal indecomposible} case. The following is crucial




\begin{lem}[Simplex]\label{simplex}
When the action is maximal indecomposible, $C$ is an $n$-simplex.
\end{lem}

\begin{proof}
Consider (any) $n+1$ faces $F_1, \ldots, F_{n+1}$. First note that any $n$ of them  intersect in exactly one point. If say, e.g.,  $ F_1 \cap \ldots \cap F_n$ contains at least two points, then (each component of) the 1- dimensional strata, e.g., $F_1 \cap \ldots \cap F_{n-1}$ is a geodesic joining two points of $ F_1 \cap \ldots \cap F_n$. Now $F_{n+1}$ must intersect at least one of these geodesics at an interior point, which is clearly impossible. Thus $p_i =  F_1 \cap \ldots \cap \hat {F_i} \cap \ldots \cap F_{n+1}, i = 1, \ldots , n+1$ are $n+1$ vertices of $C$. Now suppose there is another face $F_{n+2}$. Using the same reasoning it follows that the $n+1$ intersections of any $n$ among $F_{n+2}, F_1, \ldots, F_n$ coincide with the vertices $p_i, i=1, \ldots, n+1$. This on the other hand is impossible unless $F_{n+2} = F_{n+1}$, i.e., $C$ has exactly $n+1$ faces.

Now consider a vertex, say $p_{n+1}$ and its opposite face $F_{n+1}$. From Lemma 5.1 in \cite{Wo} we immediately get that $p_{n+1}$ is the set at maximal distance to $F_{n+1}$, in particular it is the soul of $C$ constructed from dist$ (F_{n+1}, \cdot)$. Using that all non-maximal super level sets of dist$ (F_{n+1}, \cdot)$ are convex we construct (applying a standard partitian of unity argument starting inductively at the most singular strata involving $F_1, \ldots, F_n$ and then $F_{n+1}$) a smooth gradient like vector field on $C - \{p_{n+1}\}$ which is tangent to all strata $F_1, \ldots, F_n$, radial near $p_{n+1}$ and transverse to $F_{n+1}$. Since a small ball around $p_{n+1}$ in $C$ is clearly a simplex, this competes the proof.
\end{proof}

\begin{rem}
An alternative proof of the above claim using only Riemannian geometry, i.e., not appealing to \cite{Wo} can be carried out by considering the convex $J = \{1,\ldots,n\}$ residue  $A_{p_{n+1}}:=\W_{p_{n+1}}C = \cup _{w\in \W_{p_{n+1}}} wC$ of $C$ in $M$, where  $\W_{p_{n+1}} = \W_J$ is the isotropy group of $p_{n+1}$.  Note, that $\partial A_{p_{n+1}}$ is the union of faces opposite $p_{n+1}$, and that $A_{p_{n+1}}$ has smooth totally geodesic interior, with $p_{n+1}$ an interior point. Now one applies Riemannian convexity arguments as in the soul theorem in a $\W_{p_{n+1}}$  equivariant fashion, which eventually leads to the conclusion that the soul of $A_{p_{n+1}}$ is $\{p_{n+1}\}$.
\end{rem}

Note that if all mirrors meet in the Coxeter cover $M'$ of $M$, they certainly meet in $M$ as well. So the assumptions in Theorem A in particular imply that its Coxeter chamber by the above lemma is a simplex. Thus the following Theorem and its Corollary will complete the proof of half of Theorem A in the introduction.

\begin{thm}[Spherical space form]\label{ssf}
Let $(M,\W)$ be a compact nonnegatively curved Coxeter manifold  with finite fundamental group and chamber $C$ a simplex. Then $M$ admits a $\W$ invariant metric of constant curvature $1$.
\end{thm}

\begin{proof}
 By \ref{w-structure}, $\pi_1(M)$ is a normal subgroup of the Coxeter group $\hat{\W}$ acting on the universal cover $\tilde{M}$. By assumption $\tilde{M}$ is compact and hence $\hat{\W}$ is a finite Coxeter group, acting simply transitively on its set of chambers. In this case all tangent cones at any point of a
boundary strata of a chamber $\tilde C = \tilde{M}/\hat{\W}$ ($=C$) is isometric to a
corresponding tangent cone for the linear action by the Coxeter group
$\hat{\W}$ on $\Sph^n$. From the above lemma and arguing as in Corollary
2.10 of \cite{FGT} we see that $\tilde C$ admits a metric of constant
curvature 1, which extends via $\hat{\W}$ to an invariant metric on $\tilde{M}$. Since $\pi_1(M) \lhd \hat{\W}$ one gets an induced constant curvature metric on $M$ invariant under $\W$.
\end{proof}

In the above theorem it is well known that $\pi_1(M) \lhd \hat{\W}$ is either trivial or $\Z_2$ acting as the antipodal map on the sphere. This has the following somewhat surprising consequence:

\begin{cor}\label{sphere lift}
Let $M$ be compact non negatively curved manifold with finite fundamental group. Then any reflection group $\W$, whose Coxeter lift has simplex chambers admits an invariant metric of constant curvature.
\end{cor}

\begin{proof}
From \ref{ssf} and \ref{w-structure}, we know that the Coxeter cover $M' = \hat{M}$ is either $\Sph^n$ or $\RP^n$ with an invariant metric of nonnegative curvature. Recall that by construction of the Coxeter cover in \ref{coxeter} the action by $\Gamma' \cong \W_c$ on $M'$ commutes with the $\W$ action. In particular, any mirror as well as its complement is preserved by $\Gamma'$. If $M'$ is  $\RP^n$ such a complement is a (convex) open disc, whose soul must be a point preserved by $\W_c$. But since the action is free, $\W_c$ must be trivial.

If $M'$ is $\Sph^n$, it follows that $\Gamma' = \pi_1(M)$ commutes with $\hat{\W} = \W$. Arguing as in the projective space case, it follows that $\W_c \cong \Gamma'$ is either trivial or $\Z_2 = \langle a \rangle $. In the latter case note that $a$ induces an automorphism of $\W$ which is reflected also in the induced action, by say $A$ on the orbit space simplex $\Sph^n/\W$. Now $A$ fixes the soul point of the simplex $\Sph^n/\W$ and maps vertices according to the induced automorphism of the diagram for $\W$. Now using convexity and critical point theory arguments
$\langle A\rangle $-equivariantly, we conclude that $\hat C = \Sph^n/\W$ admits an $A$-invariant metric of constant curvature, analogous to the proof of Corollary 2.10 of \cite{FGT}, just like its linear model. It follows, that $M$ admits an invariant constant curvature metric, i.e., $M = \RP^n$ and $\W_c \cong \Gamma' = \pi_1(M)$ acts on  $\Sph^n$ as the antipodal map.
\end{proof}

\begin{rem}
We will see that when $M$ has non-compact universal cover, and the action is indecomposable, then sections are flat. In particular, $C'$ is a flat simplex  when the action is maximal indecomposable. This will lead to a proof of the sencond half of Theorem A in the introduction  (cf. the Torus Theorem \ref{torus}).
\end{rem}

\section{Open book structures: All faces meet}   

In this section we will develop complete structure results for  Coxeter manifolds of nonnegative curvature, where all mirrors meet,  equivalently the chamber $C = \Bbb V$ in the splitting \ref{base split}. These Theorems will have Theorem B of the introduction as an immediate consequence.
\smallskip

As in the case of maximal indecomposible  Coxeter actions it is crucial to understand the structure of a chamber $C$. Note that in the case under consideration, there are $k \le n$ mirrors in $M$ and their intersection coincides with the fixed point set $M^{\W}$, which in $C$ is also the intersection $B:= F_1 \cap \ldots \cap F_k$
of all its faces, i.e., $\hat{F}_{\emptyset}$ recalling our notation from \ref{chamber} to be used throughout below. Moreover, for each of $p \in B$, $\W$ acts effectively on the normal sphere $\Sph^{\perp} = \Sph^{k-1} \subset \R^k$ to $B$ at $p$, and $\W$ is a finite Coxeter group.

\medskip
For each face $F_i$ we let $S_i \subset C$ be the soul in $C$ associated to the distance function $d_i: = \text{dist}( F_i,  \cdot )$ to $F_i$. Recall that  by construction, $S_i$ is the image, $Sh_i(C)$ of the associated Sharafutdinov deformations retraction, $Sh_i : C \to C$ of $C$. Since, this retraction is a concatenation of \emph{gradient pushes}, and gradient pushes preserve extremal sets \cite{Pe} it follows immediately that

\begin{itemize}
\item
For each $i \in I$, the soul $S_i$ meets every component of  $\hat{F}_i$.
\end{itemize}

In particular, if $S_i \subset \hat{F}_i$ it follows that $\hat{F}_i$ is connected. Moreover,

\begin{lem}[Reduction]\label{reduction}
If $S_i$ is not a subset of $\hat{F_i}$, it is perpendicular to it, and the normal slice representation of $\W$ along $B$ is reducible.
\end{lem}

\begin{proof}
Consider the $I-i$ residue $\W_{\hat{F_i}}C$ of $C$ with boundary $\W_{\hat{F_i}}F_i$. Clearly, the usual Riemannian construction of the soul of  $\W_{\hat{F_i}}C$ is $\W_{\hat{F_i}}$ invariant and equivalent to working on $C$. In particular, its soul is $\W_{\hat{F_i}}S_i$, a totally geodesic sub manifold of  $\W_{\hat{F_i}}C$. Pick a point $p \in S_i \cap \hat{F}_i \subset \W_{\hat{F_i}}S_i$. Obviously, the tangent space to the soul $\W_{\hat{F_i}}S_i$ of $\W_{\hat{F_i}}C$ at $p$ is $\W_{\hat{F_i}}$ invariant. If this is not a subspace of the tangent space to $\hat{F_i}$, its complement is perpendicular to it, i.e., $S_i$ is perpendicular to $\hat{F_i}$.

If $S_i - \hat{F}_i  \ne \emptyset$, there is a smallest strata $D = \hat{F}_J$ containing $S_i$ and meeting it at interior points of $D$. Suppose first that $D = C$ and let $pq$ be a minimal geodesic from $p \in c \cap S_i$ to  $S_i \cap  \hat{F_i}$, and $pq_i$  a minimal geodesic from $p$ to $F_i$. Clearly, $pq_i$ is perpendicular to $pq$ as well as to $F_i$. It follows that $pq_i$ and $pq$ are adjacent edges in an isometrically embedded flat rectangle in $C$ with opposite edges in $F_i$, respectively a minimal geodesic $\gamma$ from $S_i \cap  \hat{F_i}$ to $F_i$. Since $qp$ is not on the boundary of the normal space of directions  to $\hat{F_i}$ in $C$ at $q$, it follows that $\gamma$ is a geodesic in $\hat{F_i}$, and in particular we see that the $\W$ normal slice representation is reducible.

In general, if $S_i$ is not contained in $\hat{F}_i$, let $D= \hat{F}_J$ be the smallest strata containing $S_i$ and meeting it at interior points. In the residue,  $\W_{\hat{F_i}}C$ consider the corresponding totally geodesic subset $M^{\W_{I-J}} \cap \W_{\hat{F_i}}C$, i.e., the intersection of the residue with the mirrors determined by $\hat{F}_J$. Clearly the soul of the residue is contained in this subset. Moreover, the Sharafutdinov retraction of the  residue preserves the subset, and since it is totally geodesic, this restricted deformation retraction is also distance non-increasing with respect to the intrinsic metric on the set. From these properties, it follows as in the original approach by Sharafutdinov (cf. also \cite{Yi1}) that the intrinsic soul of the subset is isometric to ``extrinsic" soul, i.e., the soul of the residue. Again by invariance, it follows that the intrinsic soul of the strata $D$ is isometric to $S_i$, and in particular intersects $D$ at interior points. The proof is now completed as above.
\end{proof}

We are now ready to describe the structure of $C$, when the action $\rho$ of $\W$ on the normal spaces $\R^k$ to $B$ is irreducible.

Note that for each $i$, the strata $\hat{F_i}$ is a connected, compact non negatively curved manifold with boundary $B$. Moreover, when $S_i$ is contained in $\hat{F_i}$ it follows (as in the proof above) that $\hat{F_i}$ has the structure of a disc bundle of a non negatively curved vector bundle $\nu_i$ over $S_i$.

Using the description in the model examples of the introduction, we will show that for each $i$, $M$ is equivariantly equivalent to the open book  $M_{\rho,\nu_i} = \Sph (\nu_i \oplus \varepsilon ^{k}) =:\Sph(\nu_i,\rho) \to \Disc^k$. To do this, we will show that $C = M/\W$ is  $\Sph (\nu_i \oplus \varepsilon ^{k})/\W  \to \Disc^k/\W = \text{cone}(\Delta_s^{k-1})$, where $\Delta_s^{k-1} = \Sph^{k-1}/\W$. Due to this description, we also say that $C = \Sph (\nu_i \oplus \varepsilon ^{k})/\W$ is a book with binding $B$ and pages $\Disc(\nu_i) = \hat{F_i}$ parametrized by $\Delta_s^{k-1}$. Indeed, with $\Delta_s^{k-2}$ being the space of directions in $\Delta_s^{k-1}$ opposite its face labelled $i$, we have:

\begin{lem}[Book chamber]\label{book chamber}
If the action by $\W$ on normal spaces to $B$ is irreducible, and $\W$ has rank $k \le n$, then $C$ has the structure of the ball bundle of $\nu_i \times \emph{cone}(\Delta_s^{k-2})$.

Alternatively, $C$ has the structure of a book with binding $B$ and pages $\hat{F_i}$ parametrized by $\Delta^{k-1}$, the normal space of directions in $C$ along $B$.
\end{lem}

\begin{proof}
Since by the reduction lemma \ref{reduction}, the soul $S_i \subset \hat{F_i}$ for $d_i: = \text{dist}( F_i,  \cdot )$ is a totally geodesic sub manifold, and also a soul for $\hat{F_i}$, the first claim is an immediate consequence of the soul construction for $C$ (or alternatively for its $I-i$ residue), recalling that the normal bundle of $C$ to $\hat{F_i}$ along interior points (including $S_i$), is  spanned by $k-2$ parallel fields (see remark \ref{chamber}).

For the same reason, $F_i$ is the ``sphere bundle " boundary of this bundle, i.e., the normal space of directions bundle of $S_i$ in $C$. Note, that each fiber of this bundle is the join of a normal sphere to $S_i$ in $\hat{F_i}$ with $\Delta_s^{k-2}$. Moreover, the bundle orthogonal to $\hat{F_i}$ is trivial.

Viewing the join $\Sph^m * \Sph^{k-2} \to \Disc^{k-1} = \text{cone}(\Sph^{k-2})$ as a $k-2$ dimensional \emph{open book} with \emph{binding} $\Sph^m$ and pages $\Disc^{m+1}$ (following the flow lines for the gradient of  the distance function to either sphere), provides the desired induced structure on  $\Sph^m * \Delta^{k-2} \subset \Sph^m * \Sph^{k-2}$. This in turn yields the book structure on $F_i$ and then for all of $C$, using that the simplex bundle over $S_i$  is trivial. Specifically, one constructs (much like in \cite{GK}) a smooth vector field on $C$ which is radial near $B$, tangent to all strata and transverse to the sub manifold with corners, $S_i \times \Delta^{k-1}$, that emerged from the soul construction. Using that all normal bundles to all strata are trivial as observed in \ref{chamber}, one can also arrange that $S_i \times \Delta^{k-1} \subset C$ is perpendicular to all strata.
\end{proof}

The following is now a simple consequence of the fact that all constructions above can be carried over to $M$ equivariantly, and noting the same structure on the sphere bundle $\Sph(\nu_i \oplus \varepsilon ^{k})$ of $\nu_i \times \R^k$ equipped with the obvious action by $\W$.

\begin{thm} [Open Book]\label{open book} Let $(M, \W)$ be a non negatively curved Coxeter $n$-manifold with rank $k\le n$, where all mirrors meet in $B = M^{\W}$. If the normal representation $\rho$ along $B$ is irreducible, then $M$ is equivariantly diffeomorphic to $\Sph( \nu \oplus \varepsilon ^{k})$, where $\nu$ is a non negatively curved vector bundle with sphere bundle $B$.

Alternatively, $M$ is an open book with binding $B$ and non negatively curved pages $\Disc(\nu)$ parametrized by $\Sph^{k-1}$.
\end{thm}

\begin{rem}[Converse]\label{converse} We point out that conversely we can construct a $\W$ invariant metric with nonnegative curvature on a manifold with these data. To do this we use the open book description of $\Bbb S(\nu \oplus \varepsilon ^{k})=\Sph (\nu )\times \Bbb D^k \cup
_\partial \Bbb D(\nu )\times \Sph^{k-1}$, where $\Sph (\nu
)\times \Bbb D^k $ is a small tubular neighborhood of $B= \Bbb S(\nu ) \subset \Bbb S(\nu \oplus \varepsilon ^{k})$, and the $\W$-action can be written as the gluing of the linear actions on each piece.  By \cite{Gu} we can modify the metric on $\nu$ so that it is product near infinity.
We take product metrics on $\Bbb D(\nu) \times \Bbb S^{k-1}$ and on $\Sph (\nu)\times \Bbb D^k$, where the metric on $\Bbb D^k$ is also a product near the boundary. The desired claim follows.
\end{rem}

\begin{rem}
Note that $M$ is a sphere, if the soul $S_i$ is a point, and in this case the action is a suspension or iterated suspension of the irreducible normal sphere action by $\W$.

Also, $B$ can have at most two components, and if it does, $\nu$ is a trivial line bundle, $S_i$ is isometric to $B$, and $M = \S_i\times \Sph^k$. If in this case, $S_i$ is not a point, $M$ is actually metrically a product, and the action on the second factor is the suspension of the normal sphere action by $\W$, and the metric is invariant with nonnegative curvature (cf. remark about heavens below). In particular, the orbit space $M/\W = B \times (\Sph^k/\W)$ splits, but this is the only decomposable case where $n$ faces meet and the normal action to the binding is irreducible.
\end{rem}

\begin{rem}[Spherical heavens]\label{heavens}
In the situation of the open book theorem any two souls $S_i$ and $S_j$ obviously have the same homotopy type, namely that of $C$. In fact, since the corresponding Sharafutdinov maps are distance non-increasing deformation retractions of $C$ and the souls are closed manifolds, it in fact follows as in \cite{Yi1} that they are isometric.

 When the souls are not points, it turns out that in fact much more structure compatible with the open book description is present. This is because there is a large family of choices for "Sharafutdinov retractions". To explain this, note that for any $J$-tuple of nonnegative numbers $a_J = \{a_i \ge 0\}_{i\in J}$ the subset
\begin{center}
$C^{a_J} = \{p \in C \ | \ \text{dist}(p, F_i) \ge a_i, i\in J\}$
\end{center}

\no is convex when non-empty. Clearly, this set can be obtained from $C$ by applying various "partial Sharafutdinov retractions". Moreover, when non-collapsed, i.e., dim$C^{a_J } = n$ one can consider any union or intersection of its $k$ faces as in the case of $C$ where $a_J$ is the trivial $J$ tuple. For such non-collapsed convex sets $C^{a_J}$ we can utilize further Sharafutdinov retractions associated to any face or union of faces of it. It follows, that all souls obtained in this fashion are isometric. Even more, the arguments of \cite{Yi2} carry over to our case verbatim since they pivot only around distance non increasing deformation retractions on convex subsets of a Riemannian manifold. As a result, $M$ contains a totally geodesic \emph{spherical heaven}, $\mathcal{H}$ of pseudo souls isometric to the product of $S$ with a non negatively curved metric on an $\ell$-sphere, where $\ell \ge k-1$ is the dimension of the flat trivial sub bundle of the normal bundle $S^{\perp}$ to $S$ in $M$ spanned by all parallel fields.  Here $\ell = k-1$ when $\nu$ has a unique soul, in which case, the heaven $\mathcal{H} = S \times \Sph^{k-1}$ provides a canonical $\W$-invariant ``edge" of the open book opposite its binding. When $\ell > k-1$, $\nu = \nu_0 \oplus \varepsilon^{\ell-k+1}$ and the heaven intersects the binding in a product of $S$ with a nonnegativey curved metric on an $\ell-k$ sphere, and in this case $M = \Sph( \nu_0 \oplus \varepsilon^{\ell-k+1} \oplus  \varepsilon ^{k}) =  \Sph( \nu_0 \oplus \varepsilon^{\ell+1})$ of course also has an open book structure with binding the sphere bundle of $\nu_0$, and pages the disc bundle of $\nu_0$ parametrized by an $\ell$-sphere. In the latter description, however we do not know if $\nu_0$ supports a non negatively curved metric.
\end{rem}

\smallskip

It remains to consider, the situation where the action by $\W$ on the normal spaces to $B$ is reducible. In this case, $\W=\W_1\times \cdots \times
\W_\ell$, acts in a component-wise fashion on the normal sphere $\Sph^\perp =\Sph (\Bbb
R^{k_1})*\cdots *\Sph(\Bbb R^{k_\ell})$. We point out here that in our formulation below, the component of  the $\W_i$ action is not necessarily required to be irreducible.

Although we are primarily interested in the indecomposable case, we point out that the product $(M_1\times \ldots \times M_{\ell}, \W_1\times \ldots \times \W_{\ell})$ of any $\ell$ irreducible non maximal indecomposable actions $(M_i,\W_i)$, $i = 1, \ldots, \ell$ provides a decomposable example where all mirrors meet. Our description below will include this.

Before, formulating our result, we elaborate further on the notion of an \emph{iterated open book}, which is based on having leaves being manifolds with corners:

\no Suppose for example, $P$ is a manifold with corners, the most ``singular" having local type $\R^{n-k}\times \R_+^k$. An $\ell_1$ dimensional open book with pages $P$ and binding a manifold with corners of type $\R^{n-k}\times \R_+^{k-1}$ will then be a manifold with corners of type $\R^{n-k+\ell_1}\times \R_+^{k-1}$, i.e., having decreased the corner type by one. This way, $k$ open book iterations results in a manifold without corners. Note that a  $k$-fold iterated open book $N$ has a \emph{page map} $L = (L_1, \ldots, L_k) : N \to \Disc^{m_1} \times \ldots \times \Disc^{m_k}$ with $k$ coordinates, where a page is of the form $L^{-1}(I^k)$, where each factor $I$ is a radial line in the corresponding disc. We will refer to $L^{-1}(0, \ldots, 0)$ as the \emph{pivot binding} of the iterated book.

A special case of this arises as described in the model examples of the introduction:

Given $\ell$ linear representations $\rho_i$ of finite Coxeter groups $\W_i$ on $\R^{k_i}$, and $\ell$ smooth vector bundles $\nu_i$ with base $S$. The obvious $\W = \W_1\times \ldots\times\W_{\ell}$ action  on the product of the bundles $\nu_i \oplus \varepsilon^{k_i}$ induces a $\W$ action on the \emph{fiber product}, $M_{\bar{\rho}, \bar{\nu}} := \Sph(\bar{\nu},\bar{\rho})$ of the sphere bundles $\Sph(\nu_i \oplus \varepsilon^{k_i})$, i.e., the pull back  by the diagonal map $\Delta: S \to S \times \ldots \times S$ of the product of the sphere bundles $\Sph(\nu_i \oplus \varepsilon^{k_i})$. As in the case of a single representation and bundle as above, there is a canonically associated $\W$ equivariant page map $L : M_{\bar{\rho}, \bar{\nu}} \to \Disc^{k_1}\times \dots \times \Disc^{k_{\ell}}$ where $B = L^{-1}(0, \ldots,0)$ is the intersection of all mirrors for $\W$, and $P = L^{-1}([0,1]x_1, \ldots, [0,1]x_{\ell})$ for any $\bar{x} \in \Sph^{k_1-1} \times \ldots \times \Sph^{k_{\ell}-1}$ is a manifold with corners diffeomorphic to the fiber product $\Disc(\bar{\nu})$ of the disc bundles $\Disc({\nu_i})$.

\begin{thm} [Iterated open book]\label{iterated open book}
 Let $M$ be a  compact nonnegatively curved Coxeter $\W$-manifold where all mirrors meet. Then there is a splitting (allowing one factor) of the normal slice representation splits as $\bar{\rho} = \rho_1 \times \ldots \times \rho_{\ell}$ on $\Sph ^\perp =\Sph (\Bbb R^{k_1})*\cdots *\Sph(\Bbb R^{k_\ell})$, such that  $M$ is $\W$-equivariantly diffeomorphic to a fiber product, $M_{\bar{\rho}, \bar{\nu}} := \Sph(\bar{\nu},\bar{\rho})$. Moreover,  the fiber product of any of  $\nu _1, \cdots, \nu _\ell$ mutually orthogonal sub bundles, is a totally geodesic subbundle of the sum of all of them, a vector bundle with non-negative sectional curvature over a soul $S$ of the chamber $C$.

Alternatively, $M$ is an iterated open book with pivot binding $B$ and page a non negatively curved fiber product $\Disc(\bar{\nu})$ with ortogonal totally geodesic subbundles $\Disc(\nu _J)$, $J \subset \{1, \ldots, \ell\}$ with right angles at all corners along its totally geodesic  boundary strata.
\end{thm}

\begin{proof}
For a $\W$-invariant decomposition of $\Sph^\perp$, we apply Theorem \ref{open book} to the $\W_1$-action on $M$. It
follows that $M$ is $\W_1$-equivariantly diffeomorphic to a sphere
bundle $\Sph (\nu _1\oplus  \varepsilon ^{k_1})$ over $S_1$, where $S_1$  is the soul of a chamber, $C_1$ for the $\W_1$-action on $M$. As seen in the proof of Lemma 3.2, $\Disc (\nu _1)$ can be taken to be any of the strata $\hat{F}_i^1$ in $C_1$, and the fixed point set of the $\W_1$ action, i.e., the intersection of all $\W_1$ mirrors $\Lambda_i^1$, is the subbundle $\Sph(\nu_1) =: B_1 = \bigcap_i \Lambda_i^1$. Note that, since $\W_j$, $j\ne 1$ fixes  the normal spaces to $B_1$ along $B = M^{\W}$,  the totally geodesic sub manifold $\hat{\Lambda}_i^1 \subset M$ (the double of $\hat{F}_i^1$) is invariant under $\W_2 \times \ldots \times \W_{\ell}$ for any $i$. In addition, since  $\hat{\Lambda}_i^1 = \Sph(\nu _1\oplus  \varepsilon ^1)$, we are in position to complete the proof by induction.

Specifically, we note that  $\W_2$ acts on $M=\Sph (\nu _1\oplus \varepsilon ^{k_1})
\searrow S_1$ in a fiber preserving fashion commuting with the
$\W_1$-action. If $\W_2$ acts trivially on the base, then $\W_2$
acts linearly along the fiber, hence $\nu_1=\varepsilon ^{k_2}\oplus
\nu_1'$ and $M =\Sph (\nu _1'\oplus \varepsilon ^{k_1} \oplus \varepsilon ^{k_2})
\searrow S_1$ with its $\W_1\times \W_2$ action. Therefore, we may assume that, the action of each factor
$\W_i$, $i\ge 2$, is nontrivial on $S_1$, hence $\Sph (\nu _1\oplus
\varepsilon ^{k_1}) \searrow S_1$ is an equivariant $\W_2\times
\cdots \times \W_\ell$ bundle. By induction we may assume that the
soul $S_1$ is $\W_2\times \cdots \times
\W_\ell$ equivariantly diffeomorphic to a fiber product  $\Sph(\hat{\nu}_1,\hat{\rho}_1)$ of Sphere bundles  $\Sph(\nu_i \oplus \varepsilon^{k_i})$, $i\ne 1$
 over a totally geodesic submanifold $S \subset S_1$, where $S$ is the soul of a chamber of the $\W_2\times \cdots\times \W_\ell$  action on
$S_1$. In particular,  the orbit space of the $\W_2\times \cdots \times
\W_\ell$ action on $S_1$ is the fiber product of the chambers in the disk bundles
$\Bbb D(\nu _2\oplus \varepsilon ^{k_2-1})$, $ \cdots , \Bbb D(\nu
_\ell \oplus \varepsilon ^{k_\ell -1})$. Therefore, the orbit space
of the $\W$-action on $M$ is the fiber product of chambers of $\Bbb
D(\nu _1\oplus \varepsilon ^{k_1-1}), \Bbb D(\nu _2\oplus
\varepsilon ^{k_2-1})$, $ \cdots , \Bbb D(\nu _\ell \oplus
\varepsilon ^{k_\ell -1})$, where the double of $\Bbb
D(\nu _1\oplus \varepsilon ^{k_1-1})$ is  the restriction of the sphere bundle $\Sph (\nu _1\oplus \varepsilon ^{k_1}) $ to $S$. It follows that $M$ is $\W$
equivariantly diffeomorphic to the fiber product  $\Sph(\bar{\nu},\bar{\rho})$ of $\Sph(\nu _i\oplus
\varepsilon ^{k_i})$ over  $S$.
\end{proof}

\begin{rem}
We leave the  details of the proof of the  (equivalent) iterated open book statement to the reader. Here, rather than using the induction hypothesis on the soul $S_1$, one uses it on the whole $\W_2\times
\cdots \times \W_\ell$ invariant page  $\Disc(\nu_1)$. We also point out that each irreducible sub action gives rise to a coordinate page map for an open book decomposition as in Theorem \ref{open book}. All together one gets a $\W$ equivariant page map $F: M \to \Disc^{k_1} \times \ldots \times \Disc^{k_{\ell}}$ with pages as claimed. As in the case of the open book description, one gets even more geometric structure when the normal $\W$ action is reducible. For example, one gets several heavens $\mathcal{H}_i$ corresponding to the $\W_i$ sub-actions, and their (orthogonal) intersections as totally geodesic submanifolds of $M$.
\end{rem}

\no We note that

$\bullet$ The chamber $C$ is a bundle over the soul $S$ with fiber the product  $\Sph^{i_1}*\Delta_1\times \cdots \times \Sph^{i_\ell}*\Delta_\ell$.

\begin{rem}[Reconstruction]\label{reconstruction}
As in the remark \ref{converse}, the non-negatively curved metric on $M$ can be constructed from a $\W_2\times \cdots \times \W_\ell$-invariant complete metric of non-negative curvature on the vector bundle $\nu_1$ over $S_1$, by modifying the metric near infinity (cf.  \cite{Gu}) in a $\W_2\times \cdots \times \W_\ell$ invariant fashion.

 Alternatively one can use the iterated open book description to achieve this as soon as the nonnegatively curved page metrics have been modified so as to be product metrics along the boundary and its corners. This again is done inductively using \cite{Gu} combined with the information that say the disc bundles $\Disc(\nu_1)$ and the fiber product of the remaining disc bundles $\Disc(\hat {\nu}_1)$ are orthogonal totally geodesic sub bundles of the $\Disc(\bar{\nu})$, so that either one of these manifolds with corners can be used a soul of the page.
\end{rem}

Prompted by the structure emerged in this section, we raise the following questions:

\begin{problem}
Are there obstructions for the sum / quotient of two non negatively curved bundles with common soul to have nonnegative curvature?

\end{problem}

\section{Metric rigidity: Non compact universal cover} 

Our main goal in this section is to derive rigidity properties for nonnegatively curved manifolds $M$ having noncompact universal cover and supporting a cocompact reflection group. In particular, we will see that the action is indecomposable if and only if $M$ is flat with Coxeter chamber $C'$ a euclidean simplex. Moreover, in this case $M$ is either a flat torus or flat eucidean space.

We begin with the case where $M$ itself is non-compact (and complete).

By the Cheeger - Gromoll soul theorem such a manifold contains a metrically embedded, totally convex compact submanifold $S$ (a soul of $M$) whose normal bundle is diffeomorphic to $M$. Moreover, by Corollary 6.2 in \cite{CG}, $M$ splits uniquely as a product $\bar{M} \times \R^k$, where the isometry group $I(\bar{M})$ of $\bar{M}$ is compact and $I(M) = I(\bar{M}) \times I(\R^k)$. Thus in the presence of a cocompact isometric action their work immediate yields

\begin{thm}[Strong Splitting] \label{splitting}Assume $M$ is a complete open manifold of nonnegative curvature with a cocompact isometric group action. Then $M$ is isometric to a metric product ${\Bbb R}^k\times
S$, where $S$ is a soul of $M$.
\end{thm}

In particular,

\begin{cor}[Noncompact Indecomposible]
A complete open manifold $M$ with nonnegative curvature and
cocompact reflection group $\W$ is indecomposable if and only if $M$
is isometric to flat euclidean $\R^n$ and $\W$ is an affine Coxeter
group with chamber $C = M/\W$ a euclidean
$n$-simplex.
\end{cor}

Here the last claim follows from the fact, that the factors in \ref{base split} all must be euclidean simplices for any cocompact Coxeter action on $\R^k$, and that a co-compact Affine Coxeter group has orbit space a simplex if it is indecomposable, or in this case equivalently irreducible.

\smallskip

Also, for $M$ compact with infinite fundamental group we get

\begin{prop}
Let $M$ be a compact non negatively curved manifold with infinite fundamental group and reflection group $\W$. Then, the action is decomposable unless $M$ is flat.
\end{prop}

\begin{proof}
From the Cheeger-Gromoll-Toponogov splitting theorem \cite{To, CG1} we know that the universal cover of $M^n$ splits isometrically as $\R^k \times N$, where $\R^k$ is flat euclidean $k$-space,
$k \ge 1$ and $N$ is a compact simply connected nonnegatively curved manifold. Since mirrors for the lifted reflection group $\hat{\W}$ contain either an $\R^k$ factor or an $N$ factor we have that $\hat {\W} = {\hat {\W}}_{ \R^k} \times  {\hat {\W}}_{N}$, yielding a nontrivial splitting for the $\hat \W$ chamber unless $N$ is a point. The desired result follows.
\end{proof}

 \medskip

Throughout the remaining part of this section $M$ is a compact flat
manifold. We start with the following simple observation, concerning actions where the Coxeter chamber $C'$ does not contain any simplex factors in \ref{base split}:

\begin{lem} [Flat open book]\label{flat book}Assume $M$ is a compact flat manifold with a Coxeter action by a reflection group $\W$. If all mirrors meet, then $\W\cong \Bbb Z_2^{k+\ell}$ and $M$ is isometric to
$N \times _{\Bbb Z_2 ^\ell} \Bbb T^\ell
\times \Bbb T^k$, where  $\Bbb Z_2 ^\ell$ acts freely on a compact
flat manifold $N$, $\Bbb Z_2^\ell \times \W \subset I(\Bbb S^1)
^{\ell +k} $ acting componentwise on $\Bbb T^{k+\ell }=\Sph^1\times
\cdots \times \Sph^1$ by reflecions.
\end{lem}

\begin{proof}  It is clear that the intersection of mirrors is a flat manifold. Let $N$ denote a fixed point connected component.
From  \ref{open book}, respectively \ref{iterated open book} we know
that $M$ is a bundle with fiber a sphere respectively a product of
spheres over a soul. Being flat, the soul $S$ must be flat, and the
fiber must be a product of circles. Therefore, $M$ is the fiber
product of $\Bbb S^1$-bundles $\Sph(\nu _i\oplus \varepsilon )$,
where $\nu _i$, $1\leq i \leq k+\ell$, are all real line bundles
over $S$. Assume the first $\ell$ bundles are nontrivial, and
respectively the last $k$ bundles are trivial. In particular, $N$ is
a free $\Bbb Z_2^\ell$ bundle over $S$. It is clear $\W \cong \Bbb
Z_2^{k+\ell}$ acting on $\Bbb T^{k+\ell}$ by componentwise
reflections, commuting with the componentwise $\Bbb Z_2^\ell$ action
on the first $\ell$ factors (different from the $\W$ action on the
component, note that $\Bbb Z_2\times \Bbb Z_2\subset I(\Bbb S^1)$.)
The desired result follows.
\end{proof}

It follows in particular that  the action is indecomposable if and only if the chamber of its associated Coxeter
action is a euclidean simplex. Moreover, by \ref{w-structure} we know that if the Coxeter cover $M'$ of $M$ has chamber a simplex, then $\pi_1(M') \subset \hat{\W}$, and the $\hat{\W}$ chamber in $\tilde{M}$ is a simplex as well. However, as pointed out above, it then follows that the affine Coxeter group is irreducible.

Recall that,  an irreducible affine Coxeter group $\W$ of rank $m$ must be one of types $\tilde \A_m, \tilde \B_m, \tilde \CC_m,\newline \tilde \D_m $, $ \tilde \E_6,  \tilde \E_7, \tilde \E_8, \tilde \F_4, \tilde \G_2$ (cf., e.g., \cite{Bo}), and $\W=\Bbb Z^m\rtimes \W_0$, where $\W_0$ is an irreducible spherical Coxeter group, of type $\A_m$, $\B_m=\CC_m$, $\D_m,  \E_6, \E_7, \E_8, \F_4,  \G_2$. We say that a reflection group $\W$ acting on a flat manifold $M$ is {\it irreducible} if the $\hat \W$ action on $\Bbb R^m$  is irreducible.  With this terminology, we now know that the $\W$ action is indecomposable if and only if it is irreducible, if and only if its Coxeter chamber is a simplex.

\smallskip
Before proving our main result below about irreducible actions, recall that by Bieberbach's celebrated theorem,
a finite cover of $M$ is isometric to a flat torus $\Bbb T^m=\Bbb R^m/\Bbb Z^m$. Note that every isometry of $\Bbb T^m$ lifts to a lattice preserving isometry of $\Bbb R^m$, whose isometry group is $I(\Bbb R^m) \cong \Bbb R^m\rtimes \O(m)$, and vice-versa. Therefore, $I(\Bbb T^m)$ contains $\Bbb T^m$ as a normal subgroup with quotient a finite subgroup of $\O(m)$.

In view of lemma  \ref{flat book} and subsequent comments above, the following in particular completes the proof of Theorem A in the introduction:

\begin{thm}[Torus Theorem]\label{torus} Let $M$ be a compact flat manfiold with an irreducible / indecomposable reflection group action by $\W$. Then

(1) $M$ is a flat torus $\Bbb T^m$.

(2)  The $\W$ action  is Coxeter.

(3) $\W \cong \A \rtimes \W_0$, where $\A$ is a finite abelian group of rank at most $m$, and $\W_0$ is a finite irreducible spherical Coxeter group.
\end{thm}
\begin{proof}
By  Bieberbach's theorem, $M=\Bbb T^m/\G$  where $\G\subset \O(m)$ is the holonomy.  Note that $\G$ preserves the lattice $\Bbb Z^m\subset \Bbb R^m$, hence
$\G$ is also a finite subgroup of $\GL(\Bbb Z, m)$.

By section 1, $\W$ lifts to a reflection group $\hat \W\subset \tilde \W \subset I(\Bbb R^m)=\Bbb R^m\rtimes \O(m)$ such that $\tilde \W/\pi_1=\hat \W/(\pi_1\cap \hat \W)=\W$. Recall that $\hat \W=\Bbb Z^m\rtimes \W_0$, where $\W_0$ is a maximal finite subgroup of $\hat \W$, a spherical Coxeter group.
 Since $\pi_1$ is a torsion free group,  $\pi_1\cap \hat \W$ is a torsion free normal subgroup of $\hat \W$, and hence $\pi_1\cap \hat \W\subset \Bbb Z^m$ is a sublattice. In particular,
  the split epimorphism  $\hat \W\to \W_0$ induces a split epimorphism $\W =\hat \W/(\pi_1\cap \hat \W)\to \W_0$ with kernel, $\A$, a  quotient of the sublattice in $\Bbb Z^m$. Hence (3) follows.

Now we prove (1), i.e., $\G$ is trivial. Recall that $\pi_1$ is a normal extension of $\Bbb Z^m$ by $\G$. Hence the holonomy homomorphism gives an epimorphism from $\Gamma = \pi_1/\pi_1\cap \hat \W$ onto $\G\subset \O(m)$. By Corollary 1.2, $\Gamma \times \W$ acts on a flat covering space $\hat M$ of $M$, hence $\G$ commutes with $\W_0$, the image of $\hat \W$ in $\O(m)$. In particular,
every  $g\in \G$ commutes with every $\w\in \W_0 \subset \O(m) $. Therefore, the linear irreducible  Coxeter  $\W_0$  action commutes with the linear $\G$-action on $\Bbb S^{m-1}$. It follows that $\G\subset \Bbb Z_2=\langle \pm \Bbb I\rangle $, generated by the antipodal map.  If $\G=\Bbb Z_2$, then $\pi_1$ is a normal extension of
$\Bbb Z^m$ by $\Bbb Z_2$ with monodromy $-\Bbb I $. Such an extension always splits, contradicting the fact that $\pi_1$ is torsion free.

Given (1), $\tilde \W$ is an extension of $\W$ by $\Bbb Z^m$, hence, by (3), a split extension over $\W_0$ with kernel a subgroup of translations of $\Bbb R^m$. If the $\W$ action is not Coxeter, then
$\W_c$ is isomorphic to a finite subgroup of $\tilde \W$ (in fact, isomorphic to a chamber isotropy group of $\tilde \W$ on $\Bbb R^m$), hence a subgroup of a conjugate of $\W_0$ in $\tilde \W$.
Therefore, $\W_c$ is trivial, since the $\W_0$ action on $\Bbb R^m$ is Coxeter. The desired result follows.
\end{proof}

The proof above, in fact also yields the somewhat surprising statement, that if the Coxeter chamber $C'$ only contains simplex factors in its decomposition \ref{base split}, then in particular it is Coxeter (cf. \ref{noncoxeter}). Precisely we have:


\begin{cor}\label{holonomy}
Let $M$ be a compact flat manifold with a reducible reflection group
$\W$, where $\hat \W=\hat \W_1 \times \cdots \times \hat \W_k$ such that $\hat \W_i$ is
irreducible. If the chamber $C'$ is a direct product of
euclidean simplices $\Delta_1\times \cdots \times \Delta _k$, then $M=\Bbb
T^m/\G$, where the holonomy group $\G\subset \Bbb Z_2\times \cdots
\times \Bbb Z_2$ a subgroup of $\GL(\Bbb Z,m)$ consists of block
matrices with $i$-th block $\pm \Bbb I$. Moreover, the $\W$ action  is Coxeter.
\end{cor}

It is easy to see that, if $M$ is as in the above corollary, then it is an iterated torus bundles with structure group $\Bbb Z_2$.

The Klein bottle is the  simplest example of the above type. Specifically, we have:

\begin{rem}
Any reflection group $\W$ on a Klein bottle $\Bbb K$ is reducible. Moreover, if $C'$ is a product of intervals, then  $\W\cong    \D_{2k}\times  {\Bbb Z_2}$,  or $\D_{2k} \times \Bbb Z_2 ^2$ for $k$ odd, or $\D_{2k}\times _{\Bbb Z_2} \Bbb Z^2_2 $ for $k$ even,   where $\D_{2k}$ is the dihedral group of order $2k$, $\Bbb Z_2$ is the center of $\D_{2k}$ in the balanced product.
\end{rem}
The first assertion follows immediately from 4.5.  Note that $\Bbb K$  is the quotient $\Bbb T^2/\langle \gamma \rangle$, where the involution $\gamma $ is given by $(x, y)\mapsto (-x, \bar y)$, with $x, y\in \Bbb S^1 \subset \C$ unit complex numbers.
From 4.6 we know that $\W $ is the quotient of the product of reflection groups on $\Bbb R^1$, hence, from 4.5 (3),  the quotient of the product of two dihedral groups $\D_{2k}\times \D_{2l}$ acting componentwise on $\Bbb S^1\times \Bbb S^1$, for some $k, l \ge 1$.  Moreover, the reflection group  $\D_{2k}\times \D_{2l}$ commutes with the deck involution $\gamma$, i.e., $w\gamma w^{-1}=\gamma$ for any $w\in \D_{2k}\times \D_{2l}$. Therefore, $l=1$ or $2$. If $k$ is even, the center of $\D_{2k}$ is $\Bbb Z_2$ generated by the antipodal map on $\Bbb S^1$, hence $\gamma \in \D_{2k}\times \D_4$.  By the assumption on $C'$ we know that, if $l=1$, then $\D_2$ is not the complex conjugation on $\Bbb S^1$. From the fact that the quotient of a dihedral group is again a dihedral group the second assertion follows.

\medskip

\section{Universal cover and group decomposition}
\medskip

Our objective in this section is to prove Theorem C and Corollary D in the introduction.

To do this assume without loss of generality that the co-compact $\W$ action on $M$ is Coxeter with chambers, $C = M/\W$. Based on the previous sections and \ref{base split} we have a metric decomposition of the form

\begin{equation}
 C =\displaystyle \prod _{i=1}^r  \Delta^e_i \times  \displaystyle \prod _{j=1}^{\ell -1} \Delta^s_{j} \times \Bbb V_\ell \times N
\end{equation}

\no where $N$ is a closed non-negatively curved manifold without boundary (possibly a point), the $\Delta^e_i $ are euclidean simplexes (including intervals), $\Delta^s_j $ are spherical simplices, and $\Bbb V$ is a (iterated) \emph{book chamber}.

We start with a simple observation

\begin{lem} [Trivial factor]\label{factor}The above $M$ is isometric to $\bar M\times N$ where $\W$ acts trivially on $N$, and $\bar M$ is a non-negatively curved Coxeter  $\W$-manifold with orbit space  as above without the $N$ factor.
\end{lem}

\begin{proof} Consider the composition of  submetries $p: M\to M/\W\to N$. This yields a horizontal and vertical splitting of the tangent bundle of $M$, both of which are integrable and totally geodesic. Clearly the fiber $\bar M$ supports an induced $\W$-action, with chamber $\bar C= \displaystyle \prod _{i=1}^r  \Delta^e_i \times  \displaystyle \prod _{j=1}^{\ell -1} \Delta^s_{j} \times  \Bbb V_\ell$. Using the decomposition (5.1) we can define an equivariant map $f: M \to \bar M\times N$ by identifying a chamber $C$ with $\bar C\times N$, a chamber for the product $\W$-action on $\bar M \times N$, where $\W$ acts trivially on $N$. It is clear that $f$ is a diffeomorphism which restricts to an isometry on every chamber $wC$, for any $w\in \W$.  The  desired result follows.
\end{proof}

The following shows that (5.2) does not hold unless the action is Coxeter.

\begin{example}\label{noncoxeter}
Consider the product action on $\Bbb S^m\times \Bbb S^n$ of a linear
irreducible Coxeter $\W$ action on $\Bbb S^m$ and the trivial action on $\Sph^n$.
Let $\Bbb S^m\times _{\Bbb Z_2} \Bbb S^n$ be the orbit space of the free diagonal antipodal involution. Then the induces $\W$-action on $\Bbb S^m\times _{\Bbb Z_2} \Bbb S^n$is not Coxeter. A chamber $C$ is isometric to $\Delta \times \Bbb S^n$, but the chamber isotropy group $\W_c=\Bbb Z_2$ acts freely on the product with orbit space $\Delta \times \Bbb P^n$.

Note, that this example may be modified by replacing the antipodal map on the $\Sph^n$ factor by any isometric involution $a$. In particular, if we take $n=m=1$, $\W = \A_2$ and $a=r$ a reflection, we get a non-Coxter action on the Klein bottle, with chamber, $\Sph^1 \times \Delta^1$ and orbit space, an ``open envelope", i.e, the double of a flat rectangle, leaving one side open (cf. \ref{holonomy}).
\end{example}

By Lemma 5.2 we now assume $N$ is a point.
Note, that faces of $C$ are products of all factors but one, with faces of the remaining factor. Moreover, each such set $i$ of such faces, generate a reflection group $\W_i$ any two of which commute.

\begin{proof}[Proof of Theorem C]
Let us first consider the case where $\pi_1(M)$ is infinite. Then
by the Cheeger-Gromoll splitting theorem,
the universal cover $\tilde M$ is isometric to the product $\Bbb R^k \times N$, where $N$ is a compact simply connected manifold. Clearly, the chamber $\tilde{C}$ for the lifted $\hat{\W}$ action is a product of euclidean simplices with a chamber $C_N$ in $N$, and $\hat{\W} = \hat{\W}_0 \times \hat{\W}_N$, where $ \hat{\W}_0$ is an affine Coxeter group, and $\hat{\W}_N$ is a finite Coxeter group.

In particular, it remains to prove the claim when $\pi_1(M)$ is finite. Thus it suffices to consider that case where $M$ is compact and simply connected. In this case, there are no euclidean simplices in the splitting of $\tilde C$, and an open book chamber is simply connected as well. The splitting of the tiles, by equivariance, obviously gives rise to a local hence global splitting of $M$ into factors consisting of spheres and an open (iterated) book as claimed, with corresponding actions of Coxeter groups.
\end{proof}

\begin{proof}[Proof of Corollary D] By Theorem C, passing to the universal cover, $\tilde M$, the lifted reflection group $\hat \W$ is a product $\hat \W_0\times \hat \W_1\times \cdots \times \hat \W_\ell$, where $\hat \W_0$ is an affine Coxeter group,  $\hat \W_j$, $1\leq j \leq \ell$, are finite spherical Coxeter groups.   Note that $\W=\hat \W/\N$, where $\N$ is a normal subgroup in $\hat \W$ acting freely on $\tilde M$, as a subgroup of the deck transformations. It suffices to prove that $\N$ is abelian. Note, that $\N$  clearly projects to a  normal subgroup $p_j(\N)\subset \hat \W_j$, and moreover, $\N$ is contained in the product of $p_0(\N)\times \cdots \times p_\ell(\N)$. Hence it remains only to show that $p_j(\N)$ is abelian.

Note that $p_j(\N)$ acts freely on the $j$-th factor. Therefore, $p_0(\N)\subset \hat \W_0\cong \Bbb Z^m\rtimes \W_0$  is contained in the torsion free lattice (cf. Theorem \ref{torus}).  A spherical factor $\hat \W_j$ of rank $2$, must come from either an open book factor or a factor acting linearly on a sphere of dimension at least $2$. In either case $\hat \W_j$ has a fixed point, and hence, $p_j(\N)$ must be trivial. Finally, from the well-known fact that a normal subgroup of an irreducible spherical Coxeter group of rank at least $3$ is contained in its center (trivial or $\Bbb Z_2$) the desired result follows.

Conversely, for an abelian normal subgroup $\N \cong \Bbb Z^p \times \Bbb Z_2^q\lhd \hat \W_0\times \hat \W_1\times \cdots \times \hat \W_\ell$, where $\Bbb Z_2^q$ is in the center of the product of spherical Coxeter groups, which acts freely on the product of spheres $\Bbb S^{k_1}\times \cdots \times \Sph^{k_\ell}$, as a sub-action of the product of the antipodal maps. Therefore $\N$ acts freely on the product $\Bbb R^k \times \Sph^{k_1}\times \cdots \times \Sph^{k_\ell}$, and $\W$ acts as reflection groups on the quotient space, a manifold with non-negative curvature. The proof is now complete.
\end{proof}

\end{document}